\documentclass[a4paper,12pt]{amsart}
\title[$\boldsymbol{T}$-algebra homomorphisms]{$\boldsymbol{T}$-algebra homomorphisms between rational function semifields of tropical curves}

\author{Song JuAe}
\address{Tokyo Metropolitan University, 1-1 Minami-Ohsawa, Hachioji, Tokyo, 192-0397, Japan.}
\email{songjuae@tmu.ac.jp}


\usepackage{amsmath,amssymb,amscd}
\usepackage{amsthm}
\usepackage{color}
\usepackage{subfigure}
\usepackage{comment}
\usepackage{mathrsfs}

\newtheorem{dfn}{Definition}[section]
\newtheorem{thm}[dfn]{Theorem}
\newtheorem{prop}[dfn]{Proposition}
\newtheorem{cor}[dfn]{Corollary}
\newtheorem{lemma}[dfn]{Lemma}

\newtheorem{ex}[dfn]{Example}
\newtheorem{cl}[dfn]{Claim}

\def\Gamma{\varGamma}

\begin{document}

\maketitle

\begin{abstract}
We prove that an injective $\boldsymbol{T}$-algebra homomorphism between the rational function semifields of two tropical curves induces a surjective morphism between those tropical curves, where $\boldsymbol{T}$ is the tropical semifield $(\boldsymbol{R} \cup \{ -\infty \}, \operatorname{max}, +)$.
\end{abstract}

\section{Introduction}
	\label{section1}

The main purpose of this paper is to contribute to the construction of an algebraic foundation for abstract tropical geometry;
roughly, tropical geometry is an algebraic geometry over the tropical semifield $\boldsymbol{T} := ( \boldsymbol{R} \cup \{ -\infty \}, \operatorname{max}, +)$.
It has been studied for about two decades.
An operation called \textit{tropicalization} is a basic interest because it maps $n$-dimensional algebraic varieties to $n$-dimensional polyhedral complexes called tropical varieties that capture many important properties of the algebraic varieties.
As research proceeds, abstract tropical geometry that studies abstract tropical varieties appeared and its intrinsic structure also has been drawing interest in recent years.
For the algebraic aspect, a tropical scheme theory is developed in \cite{Giansiracusa=Giansiracusa} and \cite{Maclagan=Rincon} and closely related to it, a tropical Nullstellensatz for congruences is proven in \cite{Bertram=Easton} and \cite{Joo=Mincheva2}.

In \cite{JuAe2}, the author gave an affirmative answer to the question ``For two tropical curves $\Gamma_1$ and $\Gamma_2$, does a $\boldsymbol{T}$-algebra isomorphism $\operatorname{Rat}(\Gamma_1) \to \operatorname{Rat}(\Gamma_2)$ induce an isomorphism (i.e., a finite harmonic morphism of degree one) $\Gamma_2 \to \Gamma_1$?"

In this paper, we consider an answer to the question of $\boldsymbol{T}$-algebra homomorphism version, i.e., the question "For two tropical curves $\Gamma_1$ and $\Gamma_2$, does a $\boldsymbol{T}$-algebra homomorphism $\operatorname{Rat}(\Gamma_1) \to \operatorname{Rat}(\Gamma_2)$ induce a morphism $\Gamma_2 \to \Gamma_1$?":

\begin{thm}
	\label{thm1}
Let $\Gamma_1, \Gamma_2$ be tropical curves.
Let $\psi : \operatorname{Rat}(\Gamma_1) \to \operatorname{Rat}(\Gamma_2)$ be a $\boldsymbol{T}$-homomorphism between their rational function semifields.
If $\psi$ is injective, then there exists a unique surjective morphism $\varphi : \Gamma_2 \twoheadrightarrow \Gamma_1$ such that $\varphi(x^{\prime})=x$ for any $x^{\prime} \in \operatorname{Max}^{\prime}_x$.
\end{thm}

Here $\operatorname{Max}^{\prime}_x$ is defined as the set $\{ x^{\prime} \in \Gamma_2 \,|\, \psi(\operatorname{CF}(\{ x \}, l))(x^{\prime}) = 0\}$ when $x \in \Gamma_1 \setminus \Gamma_{1\infty}$ with $l \in \boldsymbol{R}_{>0} \cup \{ \infty \}$; $\{ x^{\prime} \in \Gamma_2 \,|\, \psi(\operatorname{CF}(\Gamma_1 \setminus (y, x], \infty))^{\odot (-1)}(x^{\prime}) = \infty \}$ when $x \in \Gamma_{1\infty}$ with a finite point $y$ on the unique edge incident to $x$.

Theorem \ref{thm1} has the following corollary:

\begin{cor}
	\label{cor1}
The following categories $\mathscr{C}, \mathscr{D}$ are isomorphic.

$(1)$ The class $\operatorname{Ob}(\mathscr{C})$ of objects  of $\mathscr{C}$ is the tropical curves.

For $\Gamma_1, \Gamma_2 \in \operatorname{Ob}(\mathscr{C})$, the set $\operatorname{Hom}_{\mathscr{C}}(\Gamma_1, \Gamma_2)$ of morphisms from $\Gamma_1$ to $\Gamma_2$ consists of all of the injective $\boldsymbol{T}$-algebra homomorphisms $\operatorname{Rat}(\Gamma_1) \hookrightarrow \operatorname{Rat}(\Gamma_2)$.

$(2)$ The class $\operatorname{Ob}(\mathscr{D})$ of objects of $\mathscr{D}$ is the tropical curves.

For $\Gamma_1, \Gamma_2 \in \operatorname{Ob}(\mathscr{D})$, the set $\operatorname{Hom}_{\mathscr{D}}(\Gamma_1, \Gamma_2)$ of morphisms from $\Gamma_1$ to $\Gamma_2$ consists of all of the surjective morphisms $\Gamma_2 \twoheadrightarrow \Gamma_1$.
\end{cor}

The rest of this paper is structured as follows.
In Section \ref{section2}, we give the definitions of semirings and algebras, tropical curves, rational functions and chip firing moves on tropical curves, and morphisms between tropical curves.
Section \ref{section3} contains the proofs of the assertions above.

\section*{Acknowledgements}
The author thanks my supervisor Masanori Kobayashi and Yasuhito Nakajima for their helpful comments.

\section{Preliminaries}
	\label{section2}

In this section, we recall several definitions which we need later.
We refer to \cite{Golan} (resp. \cite{Maclagan=Sturmfels}) for an introduction to the theory of semirings (resp. tropical geometry) and employ definitions in \cite{Jun} (resp. \cite{JuAe3}) related to semirings (resp. tropical curves).
The definition of morphisms between tropical curves we employ in Subsection \ref{subsection2.4} is given in \cite{Chan}.
Today it is usual that we assume a morphism between tropical curves to be (finite) harmonic (cf. \cite{Chan}, \cite{JuAe3}).
However, since the pull-back of a morphism between tropical curves is a $\boldsymbol{T}$-algebra homomorphism between the rational function semifields of these tropical curves (see the beginning of Section \ref{section3}), in our setting, it is natural to employ Chan's definition of morphisms between tropical curves in \cite{Chan}.

\subsection{Semirings and algebras}
	\label{subsection2.1}

In this paper, a \textit{semiring} is a commutative semiring with the absorbing neutral element $0$ for addition and the identity $1$ for multiplication.
If every nonzero element of a semiring $S$ is multiplicatively invertible and $0 \not= 1$, then $S$ is called a \textit{semifield}.

A map $\varphi : S_1 \to S_2$ between semirings is a \textit{semiring homomorphism} if for any $x, y \in S_1$,
\begin{align*}
\varphi(x + y) = \varphi(x) + \varphi(y), \	\varphi(x \cdot y) = \varphi(x) \cdot \varphi(y), \	\varphi(0) = 0, \	\text{and}\	\varphi(1) = 1.
\end{align*}

Given a semiring homomorphism $\varphi : S_1 \to S_2$, we call the pair $(S_2, \varphi)$ (for short, $S_2$) a \textit{$S_1$-algebra}.
For a semiring $S_1$, a map $\psi : (S_2, \varphi) \to (S_2^{\prime}, \varphi^{\prime})$ between $S_1$-algebras is a \textit{$S_1$-algebra homomorphism} if $\psi$ is a semiring homomorphism and $\varphi^{\prime} = \psi \circ \varphi$.
When there is no confusion, we write $\psi : S_2 \to S_2^{\prime}$ simply.

The set $\boldsymbol{T} := \boldsymbol{R} \cup \{ -\infty \}$ with two tropical operations:
\begin{align*}
a \oplus b := \operatorname{max}\{ a, b \} \quad	\text{and} \quad a \odot b := a + b,
\end{align*}
where $a, b \in \boldsymbol{T}$, becomes a semifield.
Here, for any $a \in \boldsymbol{T}$, we handle $-\infty$ as follows:
\begin{align*}
a \oplus (-\infty) = (-\infty) \oplus a = a \quad \text{and} \quad a \odot (-\infty) = (-\infty) \odot a = -\infty.
\end{align*}
$\boldsymbol{T}$ is called the \textit{tropical semifield}.
$\boldsymbol{B} := (\{ 0, -\infty \}, \operatorname{max}, +)$ is a subsemifield of $\boldsymbol{T}$ called the \textit{boolean semifield}.

\subsection{Tropical curves}
	\label{subsection2.2}

In this paper, a \textit{graph} is an unweighted, undirected, finite, connected nonempty multigraph that may have loops.
For a graph $G$, the set of vertices is denoted by $V(G)$ and the set of edges by $E(G)$.
A vertex $v$ of $G$ is a \textit{leaf end} if $v$ is incident to only one edge and this edge is not loop.
A \textit{leaf edge} is an edge of $G$ incident to a leaf end.

A \textit{tropical curve} is the underlying topological space of the pair $(G, l)$ of a graph $G$ and a function $l: E(G) \to {\boldsymbol{R}}_{>0} \cup \{\infty\}$, where $l$ can take the value $\infty$ only on leaf edges, together with an identification of each edge $e$ of $G$ with the closed interval $[0, l(e)]$.
The interval $[0, \infty]$ is the one point compactification of the interval $[0, \infty)$.
We regard $[0, \infty]$ not just as a topological space but as almost a metric space.
The distance between $\infty$ and any other point is infinite.
When $l(e)=\infty$, the leaf end of $e$ must be identified with $\infty$.
If $E(G) = \{ e \}$ and $l(e)=\infty$, then we can identify either leaf ends of $e$ with $\infty$.
When a tropical curve $\Gamma$ is obtained from $(G, l)$, the pair $(G, l)$ is called a \textit{model} for $\Gamma$.
There are many possible models for $\Gamma$.
We frequently identify a vertex (resp. an edge) of $G$ with the corresponding point (resp. the corresponding closed subset) of $\Gamma$.
A model $(G, l)$ is \textit{loopless} if $G$ is loopless.
For a point $x$ of a tropical curve $\Gamma$, if $x$ is identified with $\infty$, then $x$ is called a \textit{point at infinity}, else, $x$ is called a \textit{finite point}.
$\Gamma_{\infty}$ denotes the set of all points at infinity of $\Gamma$.
If $\Gamma_{\infty}$ is empry, then $\Gamma$ is called a \textit{metric graph}.
If $x$ is a finite point, then the \textit{valence} $\operatorname{val}(x)$ is the number of connected components of $U \setminus \{ x \}$ with any sufficiently small connected neighborhood $U$ of $x$; if $x$ is a point at infinity, then $\operatorname{val}(x) := 1$.
We construct a model $(G_{\circ}, l_{\circ})$ called the {\it canonical model} for $\Gamma$ as follows.
Generally, we define $V(G_{\circ}) := \{ x \in \Gamma \,|\, \operatorname{val}(x) \not= 2 \}$ except for the following two cases.
When $\Gamma$ is homeomorphic to a circle $S^1$, we define $V(G_{\circ})$ as the set consisting of one arbitrary point of $\Gamma$.
When $\Gamma$ has the pair $(T, l)$ as its model, where $T$ is a tree consisting of three vertices and two edges and $l(E(T)) = \{ \infty \}$, we define $V(G_{\circ})$ as the set of two points at infinity and any finite point of $\Gamma$.
The union of $V(G_{\circ})$ and the set of the midpoints of all loops of $G_{\circ}$ defines the \textit{canonical loopless model} for $\Gamma$.
For a point $x$ of $\Gamma$, a \textit{half-edge} of $x$ is a connected component of $U \setminus \{ x \}$ with any connected neighborhood $U$ of $x$ which consists of only two-valent points and $x$.
The word ``an edge of $\Gamma$" means an edge of $G_{\circ}$.

\subsection{Rational functions and chip firing moves}
	\label{subsection2.3}

Let $\Gamma$ be a tropical curve.
A continuous map $f : \Gamma \to \boldsymbol{R} \cup \{ \pm \infty \}$ is a \textit{rational function} on $\Gamma$ if $f$ is a constant function of $-\infty$ or a piecewise affine function with integer slopes, with a finite number of pieces and that can take the values $\pm \infty$ at only points at infinity.
For a point $x$ of $\Gamma$ and a rational function $f \in \operatorname{Rat}(\Gamma) \setminus \{ -\infty \}$, $x$ is a \textit{zero} (resp. \textit{pole}) of $f$ if the sign of the sum of outgoing slopes of $f$ at $x$ is positive (resp. negative).
If $x$ is a point at infinity, then we regard the outgoing slope of $f$ at $x$ as the slope of $f$ from $y$ to $x$ times minus one, where $y$ is a finite point on the leaf edge incident to $x$ such that $f$ has a constant slope on the interval $(y, x)$.
$\operatorname{Rat}(\Gamma)$ denotes the set of all rational functions on $\Gamma$.
For rational functions $f, g \in \operatorname{Rat}(\Gamma)$ and a point $x \in \Gamma \setminus \Gamma_{\infty}$, we define
\begin{align*}
(f \oplus g) (x) := \operatorname{max}\{f(x), g(x)\} \quad \text{and} \quad (f \odot g) (x) := f(x) + g(x).
\end{align*}
We extend $f \oplus g$ and $f \odot g$ to points at infinity to be continuous on the whole of $\Gamma$.
Then both are rational functions on $\Gamma$.
Note that for any $f \in \operatorname{Rat}(\Gamma)$, we have
\begin{align*}
f \oplus (-\infty) = (-\infty) \oplus f = f
\end{align*}
and
\begin{align*}
f \odot (-\infty) = (-\infty) \odot f = -\infty.
\end{align*}
Then $\operatorname{Rat}(\Gamma)$ becomes a semifield with these two operations.
Also, $\operatorname{Rat}(\Gamma)$ becomes a $\boldsymbol{T}$-algebra with the natural inclusion $\boldsymbol{T} \hookrightarrow \operatorname{Rat}(\Gamma)$.
Note that for $f, g \in \operatorname{Rat}(\Gamma)$, $f = g$ means that $f(x) = g(x)$ for any $x \in \Gamma$.

Let $\Gamma_1$ be a closed subset of a tropical curve $\Gamma$ which has a finite number of connected components and no connected components consisting of only a point at infinity, and $l$ a positive number or infinity.
The \textit{chip firing move} by $\Gamma_1$ and $l$ is defined as the rational function $\operatorname{CF}(\Gamma_1, l)(x) := - \operatorname{min}\{ \operatorname{dist}(\Gamma_1, x), l \}$ with $x \in \Gamma$, where $\operatorname{dist}(\Gamma_1, x)$ denotes the distance between $\Gamma_1$ and $x$.

\subsection{Morphisms between tropical curves}
	\label{subsection2.4}

Let $\varphi : \Gamma \to \Gamma^{\prime}$ be a continuous map between tropical curves.
$\varphi$ is a \textit{morphism} if there exist loopless models $(G, l)$ and $(G^{\prime}, l^{\prime})$ for $\Gamma$ and $\Gamma^{\prime}$, respectively, such that $\varphi$ can be regarded as a map $V(G) \cup E(G) \to V(G^{\prime}) \cup E(G^{\prime})$ satisfying $\varphi(V(G)) \subset V(G^{\prime})$ and for $e \in \varphi(E(G))$, there exists a nonnegative integer $\operatorname{deg}_e(\varphi)$ such that for any points $x, y$ of $e$, $\operatorname{dist}_{\varphi(e)}(\varphi (x), \varphi (y)) = \operatorname{deg}_e(\varphi) \cdot \operatorname{dist}_e(x, y)$, where $\operatorname{dist}_{\varphi(e)}(\varphi(x), \varphi(y))$ denotes the distance between $\varphi(x)$ and $\varphi(y)$ in $\varphi(e)$.

\section{Main results}
	\label{section3}

In this section, we will prove Theorem \ref{thm1} and Corollary \ref{cor1}.

We first check that the converse of Theorem \ref{thm1} holds:

\begin{prop}
	\label{prop}
Let $\Gamma_1, \Gamma_2$ be tropical curves.
If $\varphi : \Gamma_2 \to \Gamma_1$ is a surjective morphism, then the pull-back $\varphi^{\ast} : \operatorname{Rat}(\Gamma_1) \to \operatorname{Rat}(\Gamma_2); f \mapsto f \circ \varphi$ is an injective $\boldsymbol{T}$-algebra homomorphism.
\end{prop}

\begin{proof}
Since $\varphi$ is a morphism, for any $f \in \operatorname{Rat}(\Gamma_1)$, $f \circ \varphi$ is a rational function on $\Gamma_2$.
By definition, $\varphi^{\ast}$ is a $\boldsymbol{T}$-algebra homomorphism.
For $f, g \in \operatorname{Rat}(\Gamma_1)$, if $f \not= g$, then there exists $x \in \Gamma_1$ such that $f(x) \not= g(x)$.
Since $\varphi$ is surjective, there exists $x^{\prime} \in \Gamma_2$ such that $x = \varphi(x^{\prime})$.
Hence we have
\begin{align*}
\varphi^{\ast}(f) (x^{\prime}) =&~ (f \circ \varphi) (x^{\prime})\\
=&~ f(\varphi(x^{\prime}))\\
=&~ f(x)\\
\not=&~ g(x)\\
=&~ g(\varphi(x^{\prime}))\\
=&~ (g \circ \varphi)(x^{\prime})\\
=&~ \varphi^{\ast}(g)(x^{\prime}).
\end{align*}
Thus $\varphi^{\ast}$ is injective.
\end{proof}

By the following examples, we know that, in general, a semiring homomorphism between semifields may not be injective.

\begin{ex}
	\label{ex1}
\upshape{
The correspondence $\boldsymbol{T} \to \boldsymbol{B}; -\infty \not= t \mapsto 0; -\infty \mapsto -\infty$ is a noninjective semiring homomorphism.
}
\end{ex}

\begin{ex}
	\label{ex2}
\upshape{
Let $\Gamma := [0, 2]$ and $\Gamma_1 := [0, 1] \subset \Gamma$.
The natural inclusion $\iota : \Gamma_1 \hookrightarrow \Gamma$ is a nonsurjective morphism and the pull-back $\iota^{\ast} : \operatorname{Rat}(\Gamma) \to \operatorname{Rat}(\Gamma_1); f \mapsto f \circ \iota$ is the restriction map $f \mapsto f|_{\Gamma1}$ and is a surjective $\boldsymbol{T}$-algebra homomorphism which is not injective.
}
\end{ex}

By Example \ref{ex2}, we also know that a morphism between tropical curves may not be surjective and to consider the condition ``$\psi$ is injective" is fundamental.
For more details on morphisms between tropical curves, see \cite{Chan}. 

Let $\Gamma_1, \Gamma_2$ be tropical curves and $\psi : \operatorname{Rat}(\Gamma_1) \to \operatorname{Rat}(\Gamma_2)$ an injective $\boldsymbol{T}$-algebra homomorphism.
The following lemma is easy but fundamental and is proven in the same way as the proof of \cite[Lemma 3.7]{JuAe2}.
Here we put its proof for readability.

\begin{lemma}
	\label{lem4}
For any $f \in \operatorname{Rat}(\Gamma_1)$, the following hold:

$(1)$ $\operatorname{max}\{ f(x) \,|\, x \in \Gamma_1 \} = \operatorname{max}\{ \psi(f)(x^{\prime}) \,|\, x^{\prime} \in \Gamma_2 \}$, and

$(2)$ $\operatorname{min}\{ f(x) \,|\, x \in \Gamma_1 \} = \operatorname{min}\{ \psi(f)(x^{\prime}) \,|\, x^{\prime} \in \Gamma_2 \}$.
\end{lemma}

\begin{proof}
If $f \in \boldsymbol{T}$, the assertions are clear.

Assume that $f$ is not a constant function.
Let $a$ be the maximum value of $f$.
In this case, $a$ is in $\boldsymbol{R} \cup \{ \infty \}$.

Assume $a \in \boldsymbol{R}$.
For $b \in \boldsymbol{R}$, we have
\begin{align*}
f \oplus b
\begin{cases}
= b \quad \text{if } b \ge a,\\
\not= b \quad \text{if } b < a.
\end{cases}
\end{align*}
Therefore we have
\begin{align*}
\psi(f) \oplus b = \psi(f) \oplus \psi(b) = \psi(f \oplus b)\\
\begin{cases}
= \psi(b) = b \	\	\text{if } b \ge a,\\
\not= \psi(b) = b \	\	\text{if } b < a.
\end{cases}
\end{align*}
Thus the maximum value of $\psi(f)$ is $a$.

Assume $a = \infty$.
Then for any $t \in \boldsymbol{T}$, we have $f \oplus t \not= t$.
Thus 
\begin{align*}
\psi(f) \oplus t = \psi(f) \oplus \psi(t) = \psi(f \oplus t) \not= \psi(t) = t
\end{align*}
hold.
This means that the maximum value of $\psi(f)$ is $\infty$.

For the minimum values of $f$ and $\psi(f)$, we can obtain the conclusion by applying the maximum value case for $f^{\odot (-1)} = -f$ and $\psi(f^{\odot (-1)}) = -\psi(f)$ since 
\begin{align*}
\operatorname{min}\{ f(x) \,|\, x \in \Gamma_1 \} = - \operatorname{max}\{ -f(x) \,|\, x \in \Gamma_1 \}
\end{align*}
and 
\[
\operatorname{min}\{ \psi(f)(x^{\prime}) \,|\, x^{\prime} \in \Gamma_2 \} = - \operatorname{max}\{ -\psi(f)(x^{\prime}) \,|\, x^{\prime} \in \Gamma_2 \}. \qedhere
\]
\end{proof}

Now we start to prove Theorem \ref{thm1}.
The proof is broken into several steps.
The main idea to construct the map $\varphi : \Gamma_2 \to \Gamma_1$ is that we extract the imformation of the ``fibre" of $x \in \Gamma_1$ from a rational function on $\Gamma_2$ of the form of $\psi(\operatorname{CF}(\{ x \}, l))$ when $x \in \Gamma_1 \setminus \Gamma_{1\infty}$ with $l \in \boldsymbol{R}_{>0} \cup \{ \infty \}$ or $\psi(\operatorname{CF}(\Gamma_1 \setminus (y, x], \infty))^{\odot(-1)}$ when $x \in \Gamma_{1\infty}$ with a finite point $y$ on the unique edge incident to $x$.
Note that a chip firing move of the form of $\operatorname{CF}(\{ x \}, l)$ or $\operatorname{CF}(\Gamma_1 \setminus (y, x], \infty)^{\odot(-1)}$ takes its maximum value at and only at $x$.

\begin{cl}
	\label{cl1}
For any $x \in \Gamma_1 \setminus \Gamma_{1\infty}$ and $l_1, l_2 \in \boldsymbol{R}_{>0} \cup \{ \infty \}$,
\begin{align*}
\{ x^{\prime} \in 
\Gamma_2 \,|\, \psi(\operatorname{CF}(\{ x \}, l_1))(x^{\prime}) = 0 \} = \{ x^{\prime} \in \Gamma_2 \,|\, \psi(\operatorname{CF}(\{ x \}, l_2))(x^{\prime}) = 0 \}.
\end{align*}
\end{cl}

\begin{proof}
For $l$ such that $0 < l < \infty$, since 
\begin{align*}
\operatorname{CF}(\{ x \}, l) = \operatorname{CF}(\{ x \}, \infty) \oplus (-l),
\end{align*}
we have 
\begin{align}
	\label{eqn1}
\psi(\operatorname{CF}(\{ x \}, l)) = \psi(\operatorname{CF}(\{ x \}, \infty)) \oplus (-l).
\end{align}
Hence we have the conclusion by Lemma \ref{lem4}.
\end{proof}

By Claim \ref{cl1}, for $x \in \Gamma_1 \setminus \Gamma_{1\infty}$, the set $\{ x^{\prime} \in \Gamma_2 \,|\, \psi(\operatorname{CF}(\{ x \}, l))(x^{\prime}) = 0 \}$ is independent of the choice of $l \in \boldsymbol{R}_{>0} \cup \{ \infty \}$.
Let $\operatorname{Max}_x^{\prime}$ denote this set.

Similarly, we can prove the following claim:

\begin{cl}
	\label{cl2}
For any $x \in \Gamma_{1\infty}$ and finite points $y_1, y_2$ on the unique edge $e$ incident to $x$,
\begin{align*}
&~ \{ x^{\prime} \in \Gamma_2 \,|\, \psi(\operatorname{CF}(\Gamma_1 \setminus (y_1, x], \infty))^{\odot (-1)}(x^{\prime}) = \infty \}\\
=&~ \{ x^{\prime} \in \Gamma_2 \,|\, \psi(\operatorname{CF}(\Gamma_1 \setminus (y_2, x], \infty))^{\odot (-1)}(x^{\prime}) = \infty \}.
\end{align*}
\end{cl}

\begin{proof}
We can choose a finite point $y_3$ on $e$ such that each $y_1$ and $y_2$ is not farther from $x$ than $y_3$.
By Lemma \ref{lem4}, the maximum value of $\psi(\operatorname{CF}(\Gamma_1 \setminus (y_1, x], \infty))^{\odot (-1)}$ is $\infty$.
For $i = 1, 2$ and the value $a_i := \operatorname{CF}(\Gamma_1 \setminus (y_3, x], \infty)(y_i) < 0$,
\begin{align*}
a_i \odot \operatorname{CF}(\Gamma_1 \setminus (y_3, x], \infty)^{\odot (-1)} \oplus 0 = \operatorname{CF}(\Gamma_1 \setminus (y_i, x], \infty)^{\odot (-1)}.
\end{align*}
Thus we have 
\begin{align}
	\label{eqn2}
a_i \odot \psi(\operatorname{CF}(\Gamma_1 \setminus (y_3, x], \infty))^{\odot (-1)} \oplus 0 = \psi(\operatorname{CF}(\Gamma_1 \setminus (y_i, x], \infty))^{\odot (-1)}.
\end{align}
Therefore we have the conclusion.
\end{proof}

By Claim \ref{cl2}, for $x \in \Gamma_{1\infty}$, the set $\{ x^{\prime} \in \Gamma_1 \,|\, \psi(\operatorname{CF}(\Gamma_2 \setminus (y, x], \infty))^{\odot (-1)}(x^{\prime}) = \infty \}$ is independent of the choice of a finite point $y$ on $e$.
Let $\operatorname{Max}_x^{\prime}$ denote this set.

\begin{cl}
	\label{cl3}
For any $x \in \Gamma_1 \setminus \Gamma_{1\infty}$, there exists $\varepsilon > 0$ such that $\psi(\operatorname{CF}(\{ x \}, \varepsilon))$ has a constant slope on each connected component of $U^{\prime} \setminus \operatorname{Max}_x^{\prime}$ and is constant $- \varepsilon$ on $\Gamma_2 \setminus U^{\prime}$, where $U^{\prime}$ is the $\varepsilon$-neighborhood of $\operatorname{Max}_x^{\prime}$.
\end{cl}

\begin{proof}
It is clear by the equality (\ref{eqn1}) in the proof of Claim \ref{cl1}.
\end{proof}

Similarly, we have the following claim:

\begin{cl}
	\label{cl8}
For any $x \in \Gamma_{1\infty}$, there exists a finite point $y$ on the unique edge incident to $x$ such that $\psi(\operatorname{CF}(\Gamma_1 \setminus (y , x], \infty))^{\odot (-1)}$ has each boundary point of $\{ x^{\prime} \in \Gamma_2 \,|\, \psi(\operatorname{CF}(\Gamma_1 \setminus (y, x], \infty))^{\odot (-1)}(x^{\prime}) = 0 \}$ as its zero; has from each such point a constant slope; has poles at and only at each point of $\operatorname{Max}_x^{\prime}$; and has no other zeros and poles.
\end{cl}

\begin{proof}
It is clear by the equality (\ref{eqn2}) in the proof of Claim \ref{cl2}.
\end{proof}

\begin{cl}
	\label{cl4}
For $x, y \in \Gamma_1 \setminus \Gamma_{1\infty}$, if $x \not= y$, then $\operatorname{Max}_x^{\prime} \cap \operatorname{Max}_y^{\prime} = \varnothing$.
\end{cl}

\begin{proof}
Assume that $\operatorname{Max}_x^{\prime} \cap \operatorname{Max}_y^{\prime} \not= \varnothing$.
Since $x \not= y \in \Gamma_1 \setminus \Gamma_{1\infty}$, there exists $\varepsilon > 0$ such that 
\begin{align*}
\{ z \in \Gamma_1 \,|\, \operatorname{CF}(\{ x \}, \varepsilon)(z) = -\varepsilon \} \cup \{ z \in \Gamma_1 \,|\, \operatorname{CF}(\{ y \}, \varepsilon)(z) = -\varepsilon \} = \Gamma_1.
\end{align*}
We have
\begin{align*}
\operatorname{CF}(\{ x \}, \varepsilon)^{\odot (-1)} \oplus \operatorname{CF}(\{ y \}, \varepsilon)^{\odot (-1)} = \varepsilon.
\end{align*}
Thus 
\begin{align*}
\psi(\operatorname{CF}(\{ x \}, \varepsilon))^{\odot (-1)} \oplus \psi(\operatorname{CF}(\{ y \}, \varepsilon))^{\odot (-1)} = \varepsilon.
\end{align*}
On the other hand, for any $z^{\prime} \in \operatorname{Max}_x^{\prime} \cap \operatorname{Max}_y^{\prime}$,
\begin{align*}
\psi(\operatorname{CF}(\{ x \}, \varepsilon))^{\odot (-1)}(z^{\prime}) \oplus \psi(\operatorname{CF}(\{ y \}, \varepsilon))^{\odot (-1)}(z^{\prime}) = 0 \not= \varepsilon.
\end{align*}
It is a contradiction.
Thus we have $\operatorname{Max}_x^{\prime} \cap \operatorname{Max}_y^{\prime} = \varnothing$.
\end{proof}

\begin{cl}
	\label{cl9}
For $x, y \in \Gamma_{1\infty}$, if $x \not= y$, then $\operatorname{Max}_x^{\prime} \cap \operatorname{Max}_y^{\prime} = \varnothing$.
\end{cl}

\begin{proof}
Assume that $\operatorname{Max}_x^{\prime} \cap \operatorname{Max}_y^{\prime} \not= \varnothing$.
Since $x \not= y$, there exists a finite point $x_1$ (resp. $y_1$) on the unique edge incident to $x$ (resp. $y$) such that $\operatorname{CF}(\Gamma_1 \setminus (x_1, x], \infty) \oplus \operatorname{CF}(\Gamma_1 \setminus (y_1, y], \infty) = 0$.
Thus we have $\psi(\operatorname{CF}(\Gamma_1 \setminus (x_1, x], \infty)) \oplus \psi(\operatorname{CF}(\Gamma_1 \setminus (y_1, y], \infty)) = 0$.
On the other hand, for any $z^{\prime} \in \operatorname{Max}_x^{\prime} \cap \operatorname{Max}_y^{\prime}$, we have
\begin{align*}
&~ \psi(\operatorname{CF}(\Gamma_1 \setminus (x_1, x], \infty))(z^{\prime}) \oplus \psi(\operatorname{CF}(\Gamma_1 \setminus (y_1, y], \infty))(z^{\prime})\\
=&~ -\infty \oplus (-\infty)\\
=&~ -\infty\\
\not=&~ 0,
\end{align*}
which is a contradiction.
\end{proof}

\begin{cl}
	\label{cl11}
If $x \in \Gamma_{1\infty}$ and $y \in \Gamma_1 \setminus \Gamma_{1\infty}$, then $\operatorname{Max}_x^{\prime} \cap \operatorname{Max}_y^{\prime} = \varnothing$.
\end{cl}

\begin{proof}
Assume that there exists an element $z^{\prime} \in \operatorname{Max}_x^{\prime} \cap \operatorname{Max}_y^{\prime}$.
There exist $x_1 \in \Gamma_1$ and $\varepsilon > 0$ satisfying 
\begin{align*}
\{ z \in \Gamma_1 \,|\, \operatorname{CF}(\Gamma_1 \setminus (x_1, x], \infty)(z) = 0 \} \cup \{ z \in \Gamma_1 \,|\, \operatorname{CF}(\{ y \}, \varepsilon)(z) = - \varepsilon \} = \Gamma_1.
\end{align*}
Then we have
\begin{align*}
\operatorname{CF}(\Gamma_1 \setminus (x_1, x], \infty) \oplus (-\varepsilon) \odot \operatorname{CF}(\{ y \}, \varepsilon)^{\odot (-1)} = 0,
\end{align*}
and hence
\begin{align*}
\psi(\operatorname{CF}(\Gamma_1 \setminus (x_1, x], \infty)) \oplus (-\varepsilon) \odot \psi(\operatorname{CF}(\{ y \}, \varepsilon))^{\odot (-1)} = 0.
\end{align*}
On the other hand, by assumption, we have
\begin{align*}
&~ \psi(\operatorname{CF}(\Gamma_1 \setminus (x_1, x], \infty))(z^{\prime}) \oplus (-\varepsilon) \odot \psi(\operatorname{CF}(\{ y \}, \varepsilon))^{\odot (-1)}(z^{\prime})\\
=&~ - \infty \oplus (-\varepsilon) \odot 0 \\
=&~ - \varepsilon \\
\not=&~ 0,
\end{align*}
which is a contradiction.
\end{proof}

By Claims \ref{cl4}, \ref{cl9}, \ref{cl11}, the correspondence $\operatorname{Max}^{\prime}_x \ni x^{\prime} \mapsto x$ becomes a map from $\bigcup_{x \in \Gamma_1} \operatorname{Max}^{\prime}_x$ to $\Gamma_1$.
We call this map $\varphi$.

\begin{cl}
	\label{cl6}
For $x \in \Gamma_1 \setminus \Gamma_{1\infty}$, $\varepsilon > 0$ such that $\psi(\operatorname{CF}(\{ x \}, \varepsilon))$ satisfies all of the conditions in Claim \ref{cl3} and any $d$ such that $0 < d < \varepsilon$, $\bigcup_{y \in \Gamma_1 : \operatorname{dist}(x, y) = d} \operatorname{Max}_y^{\prime} \subset \{ x^{\prime} \in \Gamma_2 \,|\, \psi(\operatorname{CF}(\{ x \}, \varepsilon))(x^{\prime}) = -d \}$.
\end{cl}

\begin{proof}
For $y$ such that $\operatorname{dist}(x, y) = d$, let $\delta > 0$ such that $\psi(\operatorname{CF}(\{ y \}, \delta))$ satisfies all of the conditions in Claim \ref{cl3}.
For any positive number $\varepsilon_1 < \varepsilon$, by the equality $(\ref{eqn1})$ in the proof of Claim \ref{cl1}, $\psi(\operatorname{CF}(\{ x \}, \varepsilon_1))$ satisfies all of the conditions in Claim \ref{cl3}.
Hence, if we need, by replacing $\varepsilon$ with a smaller positive number, we can assume that $d < \varepsilon < d + \delta$.
Since 
\begin{align*}
\operatorname{CF}(\{ x \}, \varepsilon) \oplus (-d) \odot \operatorname{CF}(\{ y \}, \delta) = \operatorname{CF}(\{ x \}, \varepsilon),
\end{align*}
we have
\begin{align*}
\psi(\operatorname{CF}(\{ x \}, \varepsilon)) \oplus (-d) \odot \psi(\operatorname{CF}(\{ y \}, \delta)) = \psi(\operatorname{CF}(\{ x \}, \varepsilon)).
\end{align*}
Hence, for any $y^{\prime} \in \operatorname{Max}_y^{\prime}$, $\psi(\operatorname{CF}(\{ x \}, \varepsilon))(y^{\prime}) \ge -d$.

Assume that there exists $y^{\prime} \in \operatorname{Max}_y^{\prime}$ such that $\psi(\operatorname{CF}(\{ x \}, \varepsilon))(y^{\prime}) > -d$.
Since
\begin{align*}
&~ \operatorname{CF}(\{ y \}, d) \oplus (- d) \odot \operatorname{CF}(\{ x \}, \varepsilon)^{\odot (-1)}\\
=&~ -d \odot \operatorname{CF}(\{ x \}, \varepsilon)^{\odot (-1)},
\end{align*}
we have
\begin{align*}
&~ \psi(\operatorname{CF}(\{ y \}, d)) \oplus (- d) \odot \psi(\operatorname{CF}(\{ x \}, \varepsilon))^{\odot (-1)}\\
=&~ -d \odot \psi(\operatorname{CF}(\{ x \}, \varepsilon))^{\odot (-1)}.
\end{align*}
On the other hand, we have
\begin{align*}
&~ \psi(\operatorname{CF}(\{ y \}, d))(y^{\prime}) \oplus (-d) \odot \psi(\operatorname{CF}(\{ x \}, \varepsilon))^{\odot (-1)}(y^{\prime})\\
=&~ 0 \oplus (-d) \odot \psi(\operatorname{CF}(\{ x \}, \varepsilon))^{\odot (-1)}(y^{\prime})\\
=&~ 0\\
>&~ -d \odot \psi(\operatorname{CF}(\{ x \}, \varepsilon))^{\odot (-1)}(y^{\prime}),
\end{align*}
which is a contradiction.
\end{proof}

\begin{cl}
	\label{cl5}
For $x \in \Gamma_1 \setminus \Gamma_{1\infty}$, there exists $\varepsilon > 0$ such that $\psi(\operatorname{CF}(\{ x \}, \varepsilon))$ satisfies all of the conditions in Claim \ref{cl3} and for any $d$ such that $0 < d < \varepsilon$, $\bigcup_{y \in \Gamma_1 : \operatorname{dist}(x, y) = d} \operatorname{Max}_y^{\prime} \supset \{ x^{\prime} \in \Gamma_2 \,|\, \psi(\operatorname{CF}(\{ x \}, \varepsilon))(x^{\prime}) = -d \}$.
\end{cl}

\begin{proof}
Let $\varepsilon > 0$ be such that $\psi(\operatorname{CF}(\{ x \}, \varepsilon))$ satisfies all of the conditions in Claim \ref{cl3} and $\{ y \in \Gamma_1 \,|\, 0 < \operatorname{dist}(x, y) < \varepsilon \}$ consists of only two-valent points.
For any $d$ such that $0 < d < \varepsilon$, let $y_1 \ldots, y_{\operatorname{val}(x)}$ be all of the distinct points of $\Gamma_1$ such that $\operatorname{dist}(x, y_i) = d$.
Let $\delta > 0$ be such that each $\psi(\operatorname{CF}(\{ y_i \}, \delta))$ satisfies all the conditions in Claim \ref{cl3}.
For any positive number $\varepsilon_1 < \varepsilon$ (resp. $\delta_1 < \delta$), by the equality $(\ref{eqn1})$ in the proof of Claim \ref{cl1}, $\psi(\operatorname{CF}(\{ x \}, \varepsilon_1))$ (resp. $\psi(\operatorname{CF}(\{ y_i \}, \delta_1))$) satisfies all of the conditions in Claim \ref{cl3}.
Hence, if we need, by replacing $\varepsilon$ or $\delta$ with a smaller positive number, we can assume that $d < \varepsilon = d + \delta$.
We have
\begin{align*}
&~ \left( \operatorname{CF}(\{ x \}, \varepsilon)^{\odot (-1)} \oplus d \right)^{\odot (-1)}\\
=&~ \left( \operatorname{CF}(\{ x \}, \varepsilon)^{\odot (-2)} \oplus 2d \right)^{\odot (-1)} \odot d\\
&~  \oplus (-d) \odot \operatorname{CF}(\{ y_1 \}, \delta) \oplus \cdots \oplus (-d) \odot \operatorname{CF}(\{ y_{\operatorname{val}(x)} \}, \delta),
\end{align*}
and hence
\begin{align*}
&~ \left( \psi(\operatorname{CF}(\{ x \}, \varepsilon))^{\odot (-1)} \oplus d \right)^{\odot (-1)}\\
=&~ \left( \psi(\operatorname{CF}(\{ x \}, \varepsilon))^{\odot (-2)} \oplus 2d \right)^{\odot (-1)} \odot d\\
&~  \oplus (-d) \odot \psi(\operatorname{CF}(\{ y_1 \}, \delta)) \oplus \cdots \oplus (-d) \odot \psi(\operatorname{CF}(\{ y_{\operatorname{val}(x)} \}, \delta)).
\end{align*}

Assume that $\bigcup_{i = 1}^{\operatorname{val}(x)} \operatorname{Max}_{y_i}^{\prime} \not\supset \{ x^{\prime} \in \Gamma_2 \,|\, \psi(\operatorname{CF}(\{ x \}, \varepsilon))(x^{\prime}) = - d \}$.
There exists $z^{\prime} \in \{ x^{\prime} \in \Gamma_2 \,|\, \psi(\operatorname{CF}(\{ x \}, \varepsilon))(x^{\prime}) = - d\} \setminus \bigcup_{i = 1}^{\operatorname{val}(x)}\operatorname{Max}_{y_i}^{\prime}$.
When $z^{\prime}$ is a boundary point of $\{ x^{\prime} \in \Gamma_2 \,|\, \psi(\operatorname{CF}(\{ x \}, \varepsilon))(x^{\prime}) \ge -d \}$, there exists $w^{\prime}$ near $z^{\prime}$ such that $-d - \frac{\delta}{2} < \psi(\operatorname{CF}(\{ x \}, \varepsilon))(w^{\prime}) < -d$.
Since we can assume that $\delta$ is sufficiently small so that $\psi(\operatorname{CF}(\{ y_i \}, \delta))(w^{\prime}) = - \delta$, we have
\begin{align*}
&~ \left( \psi( \operatorname{CF}(\{ x \}, \varepsilon))^{\odot (-1)} (w^{\prime}) \oplus d ) \right)^{\odot (-1)}\\
=&~ \psi( \operatorname{CF}(\{ x \}, \varepsilon))(w^{\prime})\\
\not=&~ \psi(\operatorname{CF}(\{ x \}, \varepsilon))^{\odot 2}(w^{\prime}) \odot d\\
=&~ \left(\psi( \operatorname{CF}(\{ x \}, \varepsilon))^{\odot (-2)} (w^{\prime}) \oplus 2d \right)^{\odot (-1)} \odot d\\
&~ \oplus (-d) \odot \psi(\operatorname{CF}(\{ y_1 \}, \delta))(w^{\prime}) \oplus \cdots \oplus (-d) \odot \psi(\operatorname{CF}(\{ y_{\operatorname{val}(x)} \}, \delta))(w^{\prime}).
\end{align*}
It is a contradiction.
Hence the values such $d$ are discrete even if there exist.
Thus, if we need, by replacing $\varepsilon$ with a smaller positive number, we have the conclusion.
\end{proof}

\begin{cl}
	\label{cl12}
For $\psi(\operatorname{CF}(\Gamma \setminus (y, x], \infty))^{\odot (-1)}$ in Claim \ref{cl8} and any $z \in (y, x)$, $\operatorname{Max}_z^{\prime} \subset \{ x^{\prime} \in \Gamma_2 \,|\, \psi(\operatorname{CF}(\Gamma_1 \setminus (y, x], \infty))^{\odot (-1)}(x^{\prime}) = \operatorname{dist}(y, z) \}$.
\end{cl}

\begin{proof}
For any $z \in (y, x)$, since
\begin{align*}
&~ \operatorname{CF}(\Gamma_1 \setminus (y, x], \infty)^{\odot (-1)} \oplus \operatorname{dist}(y, z) \odot \operatorname{CF}(\{ z \}, \operatorname{dist}(y, z))\\
=&~ \operatorname{CF}(\Gamma_1 \setminus (y, x], \infty)^{\odot (-1)},
\end{align*}
we have
\begin{align*}
&~ \psi(\operatorname{CF}(\Gamma_1 \setminus (y, x], \infty))^{\odot (-1)} \oplus \operatorname{dist}(y, z) \odot \psi(\operatorname{CF}(\{ z \}, \operatorname{dist}(y, z)))\\
=&~ \psi(\operatorname{CF}(\Gamma_1 \setminus (y, x], \infty))^{\odot (-1)}.
\end{align*}
Thus, for any $z^{\prime} \in \operatorname{Max}_z^{\prime}$, we have 
\begin{align*}
\psi(\operatorname{CF}(\Gamma_1 \setminus (y, x], \infty))^{\odot (-1)}(z^{\prime}) \ge \operatorname{dist}(y, z).
\end{align*}

Assume that there exists $z^{\prime} \in \operatorname{Max}_z^{\prime}$ such that $\psi(\operatorname{CF}(\Gamma_1 \setminus (y, x], \infty))^{\odot (-1)}(z^{\prime}) > \operatorname{dist}(y, z)$.
Since
\begin{align*}
&~ \operatorname{CF}(\Gamma_1 \setminus (y, x], \infty) \oplus (-2\operatorname{dist}(y, z)) \oplus (-\operatorname{dist}(y, z)) \odot \operatorname{CF}(\{ z \}, \operatorname{dist}(y, z))\\
=&~ \operatorname{CF}(\Gamma_1 \setminus (y, x], \infty) \oplus (-2\operatorname{dist}(y,z)),
\end{align*}
we have
\begin{align*}
&~ \psi(\operatorname{CF}(\Gamma_1 \setminus (y, x], \infty)) \oplus (-2\operatorname{dist}(y, z)) \oplus (-\operatorname{dist}(y, z)) \odot \psi(\operatorname{CF}(\{ z \}, \operatorname{dist}(y, z)))\\
=&~ \psi(\operatorname{CF}(\Gamma_1 \setminus (y, x], \infty)) \oplus (-2\operatorname{dist}(y,z)).
\end{align*}
On the other hand, 
\begin{align*}
&~ \psi(\operatorname{CF}(\Gamma_1 \setminus (y, x], \infty))(z^{\prime}) \oplus (-2\operatorname{dist}(y, z)) \oplus (-\operatorname{dist}(y, z)) \odot \psi(\operatorname{CF}(\{ z \}, \operatorname{dist}(y, z)))(z^{\prime})\\
=&~ -\operatorname{dist}(y,z)\\
>&~ \psi(\operatorname{CF}(\Gamma_1 \setminus (y, x], \infty))(z^{\prime}) \oplus (-2\operatorname{dist}(y,z)),
\end{align*}
which is a contradiction.
In conclusion, for any $z^{\prime} \in \operatorname{Max}_z^{\prime}$, $\psi(\operatorname{CF}(\Gamma_1 \setminus (y, x], \infty))^{\odot (-1)}(z^{\prime}) = \operatorname{dist}(y, z)$.
\end{proof}

\begin{cl}
	\label{cl13}
For $\psi(\operatorname{CF}(\Gamma \setminus (y, x], \infty))^{\odot (-1)}$ in Claim \ref{cl8} and any $z \in (y, x)$, $\operatorname{Max}_z^{\prime} \supset \{ x^{\prime} \in \Gamma_2 \,|\, \psi(\operatorname{CF}(\Gamma_1 \setminus (y, x], \infty))^{\odot (-1)}(x^{\prime}) = \operatorname{dist}(y, z) \}$.
\end{cl}

\begin{proof}
Assume that there exists $z \in (y, x)$ such that $\operatorname{Max}_z^{\prime} \not\supset \{ x^{\prime} \in \Gamma_2 \,|\, \psi(\operatorname{CF}(\Gamma_1 \setminus (y, x], \infty))^{\odot (-1)}(x^{\prime}) = \operatorname{dist}(y, z) \}$.
For any $y_1 \in (y, z)$, since
\begin{align*}
&~ \operatorname{CF}(\Gamma_1 \setminus (y_1, x], \infty)^{\odot (-1)} \odot \operatorname{dist}(y, y_1)\\
=&~ \operatorname{CF}(\Gamma_1 \setminus (y, x], \infty)^{\odot (-1)} \oplus \operatorname{dist}(y, y_1),
\end{align*}
we have
\begin{align*}
&~ \{ x^{\prime} \in \Gamma_2 \,|\, \psi(\operatorname{CF}(\Gamma_1 \setminus (y, x], \infty))^{\odot (-1)}(x^{\prime}) = \operatorname{dist}(y, z) \}\\
=&~ \{ x^{\prime} \in \Gamma_2 \,|\, \psi(\operatorname{CF}(\Gamma_1 \setminus (y_1, x], \infty))^{\odot (-1)}(x^{\prime}) = \operatorname{dist}(y_1, z) \}.
\end{align*}
As $\psi(\operatorname{CF}(\Gamma_1 \setminus (y, x], \infty))^{\odot (-1)}$ satisfies all of the conditions in Claim \ref{cl8}, so does $\psi(\operatorname{CF}(\Gamma_1 \setminus (y_1, x], \infty))^{\odot (-1)}$ by the equality $(\ref{eqn2})$ in the proof of Claim \ref{cl2}.
Thus, by replacing $y$ with $y_1$ if we need, we can assume that $\operatorname{dist}(y, z)$ plays the role of $\varepsilon$ in Claim \ref{cl5} for $\psi(\operatorname{CF}(\{ z \}, \operatorname{dist}(y, z)))$
and that there exists a point $w^{\prime}$ of some half-edge of a point of $\{ x^{\prime} \in \Gamma_2 \,|\, \psi(\operatorname{CF}(\Gamma_1 \setminus (y, x], \infty))^{\odot (-1)}(x^{\prime}) = \operatorname{dist}(y, z) \} \setminus \operatorname{Max}_z^{\prime} \not= \varnothing$ such that $\operatorname{dist}(y, z) < \psi(\operatorname{CF}(\Gamma_1 \setminus (y, x], \infty))^{\odot (-1)}(w^{\prime}) < \frac{3}{2} \operatorname{dist}(y, z)$ and $\psi(\operatorname{CF}(\{ z \}, \operatorname{dist}(y, z)))(w^{\prime}) = - \operatorname{dist}(y, z)$.
Hence we have
\begin{align*}
&~ \left( \psi(\operatorname{CF}(\Gamma_1 \setminus (y, x], \infty))^{\odot (-1)}(w^{\prime}) \oplus \operatorname{dist}(y,z) \right)^{\odot (-1)}\\
&~ \oplus (-\operatorname{dist}(y, z)) \odot \psi(\operatorname{CF}(\{ z \}, \operatorname{dist}(y, z)))(w^{\prime})\\
=&~ \psi(\operatorname{CF}(\Gamma_1 \setminus (y, x], \infty))(w^{\prime}) \oplus (- \operatorname{dist}(y, z)) \odot (- \operatorname{dist}(y, z))\\
=&~ \psi(\operatorname{CF}(\Gamma_1 \setminus (y, x], \infty))(w^{\prime})\\
\not=&~ \psi(\operatorname{CF}(\Gamma_1 \setminus (y,x], \infty))^{\odot2}(w^{\prime}) \odot \operatorname{dist}(y, z)\\
=&~ \psi(\operatorname{CF}(\Gamma_1 \setminus (y, x], \infty))^{\odot2}(w^{\prime}) \odot \operatorname{dist}(y, z) \oplus (-\operatorname{dist}(y, z)) \odot (- \operatorname{dist}(y, z))\\
=&~ \left( \psi( \operatorname{CF}(\Gamma_1 \setminus (y, x], \infty))^{\odot (-2)} (w^{\prime}) \oplus 2\operatorname{dist}(y,z) \right)^{\odot (-1)} \odot \operatorname{dist}(y, z)\\ 
&~\oplus (- \operatorname{dist}(y, z)) \odot \psi(\operatorname{CF}(\{ z \}, \operatorname{dist}(y, z))) (w^{\prime}).
\end{align*}
On the other hand, since
\begin{align*}
&~ (\operatorname{CF}(\Gamma_1 \setminus (y, x], \infty)^{\odot (-1)} \oplus \operatorname{dist}(y,z))^{\odot (-1)} 
\oplus (-\operatorname{dist}(y, z)) \odot \operatorname{CF}(\{ z \}, \operatorname{dist}(y, z))\\
=&~ (\operatorname{CF}(\Gamma_1 \setminus (y, x], \infty)^{\odot (-2)} \oplus 2\operatorname{dist}(y,z))^{\odot (-1)} \odot \operatorname{dist}(y, z)\\
&~ \oplus (- \operatorname{dist}(y, z)) \odot (\operatorname{CF}(\{ z \}, \operatorname{dist}(y, z)),
\end{align*}
we have
\begin{align*}
&~ \left( \psi( \operatorname{CF}(\Gamma_1 \setminus (y, x], \infty))^{\odot (-1)} \oplus \operatorname{dist}(y,z) \right)^{\odot (-1)}\\
&~ \oplus (-\operatorname{dist}(y, z)) \odot \psi(\operatorname{CF}(\{ z \}, \operatorname{dist}(y, z)) )\\
=&~ \left( \psi( \operatorname{CF}(\Gamma_1 \setminus (y, x], \infty))^{\odot (-2)} \oplus 2\operatorname{dist}(y,z) \right)^{\odot (-1)} \odot \operatorname{dist}(y, z)\\
&~ \oplus (- \operatorname{dist}(y, z)) \odot \psi(\operatorname{CF}(\{ z \}, \operatorname{dist}(y, z))),
\end{align*}
which is a contradiction.
\end{proof}

By Claims \ref{cl6}, \ref{cl5}, \ref{cl12}, \ref{cl13}, $\varphi$ is continuous.

\begin{cl}
	\label{cl7}
For any $x \in \Gamma_{1\infty}$, $\operatorname{Max}_x^{\prime} \subset \Gamma_{2\infty}$.
\end{cl}

\begin{proof}
By Lemma \ref{lem4}, it is clear.
\end{proof}

\begin{cl}
	\label{cl10}
$\bigcup_{x \in \Gamma_1 \setminus \Gamma_{1\infty}} \operatorname{Max}_x^{\prime} \supset \Gamma_2 \setminus \Gamma_{2\infty}$.
\end{cl}

\begin{proof}
Assume that $\bigcup_{x \in \Gamma_1 \setminus \Gamma_{1\infty}} \operatorname{Max}_x^{\prime} \not\supset \Gamma_2 \setminus \Gamma_{2\infty}$.
Since the boundary set of $\bigcup_{x \in \Gamma_1 \setminus \Gamma_{1\infty}} \operatorname{Max}_x^{\prime}$ in $\Gamma_2 \setminus \Gamma_{2\infty}$ is not empty, for any element $z^{\prime}$ in it, there exists $z \in \Gamma_1 \setminus \Gamma_{1\infty}$ such that the boundary set of $\operatorname{Max}_z^{\prime}$ in $\Gamma_2 \setminus \Gamma_{2\infty}$ contains $z^{\prime}$.
Hence $\psi(\operatorname{CF}(\{ z \}, \varepsilon))$ takes values less than zero on any half-edge of $z^{\prime}$ not in $\operatorname{Max}_z^{\prime}$ with a positive number $\varepsilon$ by Lemma \ref{lem4}.
On the other hand, by Claim \ref{cl5}, when $\varepsilon$ is sufficiently small, all such half-edges are contained in $\operatorname{Max}_{z_1}^{\prime}$ for any point $z_1$ near $z$, which is a contradiction.
\end{proof}

\begin{cl}
	\label{cl14}
$\bigcup_{x \in \Gamma_1} \operatorname{Max}_x^{\prime} = \Gamma_2$.
\end{cl}

\begin{proof}
By Claim \ref{cl10}, $\left( \Gamma_2 \setminus \bigcup_{x \in \Gamma_1 \setminus \Gamma_{1\infty}} \operatorname{Max}_x^{\prime} \right) \subset \Gamma_{2\infty}$.
Let $x_1^{\prime}, \ldots, x_n^{\prime}$ be all of the distinct points of $\Gamma_2 \setminus \bigcup_{x \in \Gamma_1 \setminus \Gamma_{1\infty}} \operatorname{Max}_x^{\prime}$.
Let $e_i^{\prime}$ be the unique edge incident to $x_i^{\prime}$.

Let $\{ z_j^{\prime} \} \subset \Gamma_2 \setminus \Gamma_{2\infty}$ be a convergent sequence such that $z_j^{\prime} \to x_i^{\prime}$ as $j \to \infty$.
Since $\varphi$ is continuous on $\bigcup_{x \in \Gamma_1} \operatorname{Max}_x^{\prime}$ and $z_j^{\prime}$ is on $e_i^{\prime}$ for each sufficiently large number $j$, the sequence $\{ \varphi(z_j^{\prime}) \}$ is convergent.
The limit $x$ of $\{ \varphi(z_j^{\prime}) \}$ must be in $\Gamma_{1\infty}$.
In fact, if $x$ is not in $\Gamma_{1\infty}$, then there exists a sufficiently small positive number $\varepsilon > 0$ as in Claim \ref{cl5} for $\psi(\operatorname{CF}(\{ x \}, \varepsilon))$.
By Claims \ref{cl6}, \ref{cl5}, there exists a sufficiently large number $M$ such that for any $j > M$, $- \varepsilon < \psi(\operatorname{CF}(\{ x \}, \varepsilon))(z_j^{\prime}) < 0$.
This means that the sequence $\{ z_j^{\prime} \}$ is not convergent to a point at infinity, which is a contradiction.
Thus $x \in \Gamma_{1\infty}$.
We can choose $z_j^{\prime}$ whose image by $\varphi$ is on the unique edge incident to $x$ and such that $\psi(\operatorname{CF}(\Gamma_1 \setminus (\varphi(z_j^{\prime}), x], \infty))^{\odot(-1)}$ satisfies all of the conditions in Claim \ref{cl8}.
For any $w \in (\varphi(z_j^{\prime}), x)$, since
\begin{align*}
&~ \operatorname{CF}(\Gamma_1 \setminus (\varphi(z_j^{\prime}), x], \infty)^{\odot(-1)} \oplus \operatorname{dist}(\varphi(z_j^{\prime}), w) \odot \operatorname{CF}(\{ w \}, \operatorname{dist}(\varphi(z_j^{\prime}), w))\\
=&~ \operatorname{CF}(\Gamma_1 \setminus (\varphi(z_j^{\prime}), x], \infty)^{\odot(-1)},
\end{align*}
we have
\begin{align*}
&~ \psi(\operatorname{CF}(\Gamma_1 \setminus (\varphi(z_j^{\prime}), x], \infty))^{\odot(-1)} \oplus \operatorname{dist}(\varphi(z_j^{\prime}), w) \odot \psi(\operatorname{CF}(\{ w \}, \operatorname{dist}(\varphi(z_j^{\prime}), w)))\\
=&~ \psi(\operatorname{CF}(\Gamma_1 \setminus (\varphi(z_j^{\prime}), x], \infty))^{\odot(-1)}.
\end{align*}
This means that $x_i^{\prime} \in \operatorname{Max}_x^{\prime}$.
In conclusion, $\bigcup_{x \in \Gamma_1} \operatorname{Max}_x^{\prime} = \Gamma_2$.
\end{proof}

\begin{cl}
	\label{cl15}
$\varphi$ is a surjective morphism from $\Gamma_2$ to $\Gamma_1$.
\end{cl}

\begin{proof}
By the discussion so far and Claim \ref{cl14}, $\varphi$ is a surjective continuous map from $\Gamma_2$ to $\Gamma_1$.
For $x \in \Gamma_1 \setminus \Gamma_{1\infty}$, by the definition of $\operatorname{Max}_x^{\prime}$, if $\operatorname{Max}_x^{\prime} \cap \Gamma_{2\infty} \not= \varnothing$, then for each $x^{\prime} \in \operatorname{Max}_x^{\prime} \cap \Gamma_{2\infty}$, $\operatorname{Max}_x^{\prime}$ contains the segment $(y^{\prime}, x^{\prime}]$ with a finite point $y^{\prime}$ on the unique edge incident to $x^{\prime}$.
Since $\Gamma_{2\infty}$ is finite, there are only a finite number of such points $x$ by Claims \ref{cl4}, \ref{cl9}, \ref{cl11}.
Also, by Claims \ref{cl12}, \ref{cl13}, for each point $x^{\prime}$ of $\bigcup_{x \in \Gamma_{1\infty}} \operatorname{Max}_x^{\prime}$, which is in $\Gamma_{2\infty}$ by Claim \ref{cl7}, there exists a finite point $y^{\prime}$ on the unique edge incedent to $x^{\prime}$ such that any $z^{\prime} \in (y^{\prime}, x^{\prime})$ is an isolated point of $\operatorname{Max}_{\varphi(z^{\prime})}^{\prime}$.
For each of these two kinds of $x^{\prime}$, we fix one $y^{\prime}$.
The closed subset $\widetilde{\Gamma_2}$ of $\Gamma_2$ obtained from $\Gamma_2$ by removing all intervals $(y^{\prime}, x^{\prime}]$ is a metric graph by Claim \ref{cl14}.
Since $\widetilde{\Gamma_2}$ is a compact metric space, it is sequentially compact.
Let $V_1^{\prime}$ be the set of all points each whose neighborhood contains half-edges where $\varphi$ has distinct expansion factors.

Assume that $V_1^{\prime}$ is an infinite set.
Then $\widetilde{\Gamma_2}$ has infinitely many points of $V^{\prime}_1$.
Every sequence in $\widetilde{\Gamma_2} \cap V^{\prime}_1$ has a convergent subsequence $\{ x_i^{\prime} \}$.
If $x_i^{\prime} \to x^{\prime}$ as $i \to \infty$, then $\varphi(x_i^{\prime}) \to \varphi(x^{\prime})$ as $i \to \infty$ by the continuity of $\varphi$.
For $x^{\prime}$, let $\varepsilon > 0$ be as in Claim \ref{cl5}.
There exist infinitely many $i$ such that $x_i^{\prime}$ is in $\{ y^{\prime} \in \Gamma_2 \,|\, - \varepsilon < \psi(\operatorname{CF}(\{ \varphi(x^{\prime}) \}, \varepsilon))(y^{\prime}) < 0 \}$ by Claim \ref{cl6}.
Since there are only finitely many non-two-valent points on a tropical curve, we can assume that this $x_i^{\prime}$ is two-valent.
By Claims \ref{cl6}, \ref{cl5}, there exists a neighborhood $U$ of $x_i^{\prime}$ where $\varphi$ has a constant expansion factor, which is a contradiction.
Hence $V_1^{\prime}$ must be a finite set.

Let $V_2^{\prime}$ be the union of $V_1^{\prime}$ and the set of vertices of the underlying graph of the canonical loopless model for $\Gamma_2$.
Let $V$ be the union of $\varphi(V_2^{\prime})$ and the set of vertices of the underlying graph of the canonical loopless model for $\Gamma_1$.
Let $V^{\prime}$ be the union of $V_2^{\prime}$ and the boundary set of $\varphi^{-1}(V)$ in $\Gamma_2$.
Then $V$ and $V^{\prime}$ determine loopless models for $\Gamma_1$ and $\Gamma_2$ respectively.
By the constructions of $V$ and $V^{\prime}$, we can check that $\varphi$ is a surjective morphism $\Gamma_2 \to \Gamma_1$ with these loopless models.
\end{proof}

The uniqueness of such $\varphi$ is clear by the definition of $\varphi$.
In conclusion, we have Theorem \ref{thm1}.

\begin{lemma}
	\label{lem7}
For $x \in \Gamma_1 \setminus \Gamma_{1\infty}$, $\varepsilon > 0$ such that $\psi(\operatorname{CF}(\{ x \}, \varepsilon))$ satisfies all of the conditions in Claim \ref{cl3} and any $d$ such that $0 < d < \varepsilon$, $\bigcup_{y \in \Gamma_1 : \operatorname{dist}(x, y) = d} \operatorname{Max}_y^{\prime} \supset \{ x^{\prime} \in \Gamma_2 \,|\, \psi(\operatorname{CF}(\{ x \}, \varepsilon))(x^{\prime}) = -d \}$.
\end{lemma}

\begin{proof}
By the proof of Claim \ref{cl5}, numbers $d$ such that $0 < d < \varepsilon$ and $\bigcup_{y \in \Gamma_1 : \operatorname{dist}(x, y) = d} \operatorname{Max}^{\prime}_y \not\supset \{ x^{\prime} \in \Gamma_2 \,|\, \psi(\operatorname{CF}(\{ x \}, \varepsilon))(x^{\prime}) = -d \}$ are discrete even if there exist.
Let $d_1$ be the minimum number of such $d$.
Then for any $l$ such that $0 < l < d_1$, we have $\bigcup_{y \in \Gamma_1 : \operatorname{dist}(x, y) = l} \operatorname{Max}^{\prime}_y = \{ x^{\prime} \in \Gamma_2 \,|\, \psi(\operatorname{CF}(\{ x \}, \varepsilon))(x^{\prime}) = -l \}$ by the definition of $d_1$ and Claim \ref{cl6}.
Let $z^{\prime} \in \{ x^{\prime} \in \Gamma_2 \,|\, \psi(\operatorname{CF}(\{ x \}, \varepsilon))(x^{\prime}) = -d_1 \} \setminus \bigcup_{y \in \Gamma_1 : \operatorname{dist}(x, y) = -d_1} \operatorname{Max}^{\prime}_y$.
Since $\psi(\operatorname{CF}(\{ x \}, \varepsilon))(z^{\prime}) = - d_1 > -\varepsilon$ and $\psi(\operatorname{CF}(\{ x\}, \varepsilon))$ satisfies all of the conditions in Claim \ref{cl3}, $z^{\prime} \in \Gamma_2 \setminus \Gamma_{2\infty}$.
By Claim \ref{cl10}, there exists $z \in \Gamma_1 \setminus \Gamma_{1\infty}$ such that $\operatorname{Max}^{\prime}_z \ni z^{\prime}$.
Then $\operatorname{dist}(x, z) > d_1$.
By Claims \ref{cl6}, \ref{cl5}, there exists $\delta > 0$ such that $\psi(\operatorname{CF}(\{ z \}, \delta))$ satisfies all of the conditions in Claim \ref{cl3} and for any $l$ such that $0 < l < \delta$, $\bigcup_{y \in \Gamma_1 : \operatorname{dist}(z, y) = l } \operatorname{Max}^{\prime}_y = \{ x^{\prime} \in \Gamma_2 \,|\, \psi(\operatorname{CF}(\{ z \}, \delta))(x^{\prime}) = -l \}$.
Since rational functions on tropical curves, there exists $y^{\prime} \in \Gamma_2 \setminus \Gamma_{2\infty}$ near $z^{\prime}$ such that $\psi(\operatorname{CF}(\{ x \}, \varepsilon))(y^{\prime}) > -d_1$ and $\psi(\operatorname{CF}(\{ z \}, \delta))(y^{\prime}) > -\delta$.
Then we have $d_1 < \operatorname{dist}(x, z) \le \operatorname{dist}(x, \varphi(y^{\prime}))  + \operatorname{dist}(\varphi(y^{\prime}), z) = -\psi(\operatorname{CF}(\{ x \}, \varepsilon))(y^{\prime}) -\psi(\operatorname{CF}(\{ z \}, \delta))(y^{\prime}) < d_1 + \delta$.
Since we can replace $\delta$ with an infinitesimal positive number, it is a contradiction.
\end{proof}

\begin{lemma}
	\label{lemma9}
Let $V$ be as in the proof of Claim \ref{cl15} and $(G, l)$ the model for $\Gamma_1$ such that $V(G) = V$.
Let $e \in E(G)$.
For any point $x$ of $e$ other than the endpoints $v, w$ (possibly $v = w$) of $e$, let $l_x$ be the minimum of $\operatorname{dist}(x, v)$ and $\operatorname{dist}(x, w)$.
Then $\psi(\operatorname{CF}(\{ x \}, l_x))$ satisfies all of the conditions in Claim \ref{cl3}.
\end{lemma}

\begin{proof}
Assume that the assertion does not hold.
There exist $e \in E(G)$ and a point $x$ of $e$ other than its endpoint(s) such that $\psi(\operatorname{CF}(\{ x \}, l_x))$ does not satisfy some condition in Claim \ref{cl3}.
By Claim \ref{cl3} and the equality (\ref{eqn1}) in the proof of Claim \ref{cl1}, there exists $\varepsilon$ such that $0 < \varepsilon < l_x$, $\psi(\operatorname{CF}(\{ x \}, \varepsilon))$ satisfies all of the conditions in Claim \ref{cl3} and for any $l > 0$, $\psi(\operatorname{CF}(\{ x \}, \varepsilon + l))$ does not satisfy some condition in Claim \ref{cl3}.
Let $y_1, y_2$ be the distinct points such that $\operatorname{dist}(x, y_i) = \varepsilon$.
Let $\delta > 0$ be such that $\psi(\operatorname{CF}(\{ y_i \}, \delta))$ satisfies all of the conditions in Claims \ref{cl3} and $\varepsilon + \delta \le l_x$ and $\operatorname{dist}(y_1, y_2) > 2\delta$.
Since
\begin{align*}
\operatorname{CF}(\{ x \}, \varepsilon + \delta) \oplus (- \varepsilon) \odot \operatorname{CF}(\{ y_i \}, \delta) = \operatorname{CF}(\{ x \}, \varepsilon + \delta),
\end{align*}
we have 
\begin{align*}
\psi(\operatorname{CF}(\{ x \}, \varepsilon + \delta)) \oplus (- \varepsilon) \odot \psi(\operatorname{CF}(\{ y_i \}, \delta)) = \psi(\operatorname{CF}(\{ x \}, \varepsilon + \delta)).
\end{align*}
By Claim \ref{cl6} and Lemma \ref{lem7}, the slope of $\psi(\operatorname{CF}(\{ x \}, \varepsilon + \delta))$ (resp. $\psi(\operatorname{CF}(\{ y_i \}, \delta))$) on each connected component of $\{ x^{\prime} \in \Gamma_2 \,|\, -\varepsilon < \psi(\operatorname{CF}(\{ x \}, \varepsilon + \delta))(x^{\prime}) < 0 \}$ (resp. $\{ x^{\prime} \in \Gamma_2 \,|\, - \delta < \psi(\operatorname{CF}(\{ y_i \}, \delta))(x^{\prime}) < 0 \}$) coincides with the expansion factor of $\varphi$ on the connected component.
Hence there exists $z^{\prime} \in \Gamma_2$ such that $\psi(\operatorname{CF}(\{ y_i \}, \delta))(z^{\prime}) = -\operatorname{dist}(y_i, \varphi(z^{\prime})) > -\delta$ and $-\varepsilon + \psi(\operatorname{CF}(\{ y_i \}, \delta))(z^{\prime}) < \psi(\operatorname{CF}(\{ x \}, \varepsilon + \delta))(z^{\prime}) \le -\varepsilon$ for $i = 1$ or $2$.
Without loss of generality, we can assume that this $i$ is one.
Note that in this case $\psi(\operatorname{CF}(\{ y_2 \}, \delta))(z^{\prime}) = - \delta$.
For any positive integer $n$, since
\begin{align*}
&~ \left( \operatorname{CF}(\{ x \}, \varepsilon + \delta)^{\odot(-1)} \oplus \varepsilon \right)^{\odot(-1)}\\
=&~ \left( \operatorname{CF}(\{ x \}, \varepsilon + \delta)^{\odot(-n)} \oplus n\varepsilon \right)^{\odot(-1)} \odot (n-1)\varepsilon\\
&~ \oplus (-\varepsilon) \odot \operatorname{CF}(\{ y_1 \}, \delta) \oplus (-\varepsilon) \odot \operatorname{CF}(\{ y_2 \}, \delta),
\end{align*}
we have
\begin{align*}
&~ \left( \psi(\operatorname{CF}(\{ x \}, \varepsilon + \delta))^{\odot(-1)} \oplus \varepsilon \right)^{\odot(-1)}\\
=&~ \left( \psi(\operatorname{CF}(\{ x \}, \varepsilon + \delta))^{\odot(-n)} \oplus n\varepsilon \right)^{\odot(-1)} \odot (n-1)\varepsilon\\
&~ \oplus (-\varepsilon) \odot \psi(\operatorname{CF}(\{ y_1 \}, \delta)) \oplus (-\varepsilon) \odot \psi(\operatorname{CF}(\{ y_2 \}, \delta)).
\end{align*}
If $\psi(\operatorname{CF}(\{ x \}, \varepsilon + \delta))(z^{\prime}) < - \varepsilon$, then the above equality cannot hold since the left-hand side takes the value $\psi(\operatorname{CF}(\{ x \}, \varepsilon + \delta))(z^{\prime})$ at $z^{\prime}$ and the right-hand side is less than it at $z^{\prime}$.
Thus $\psi(\operatorname{CF}(\{ x \}, \varepsilon + \delta))(z^{\prime}) = - \varepsilon$.
By the above argument, for $w^{\prime} \in \Gamma_2$, if $\psi(\operatorname{CF}(\{ y_k \}, \delta))(w^{\prime}) > -\delta$ and $-\varepsilon + \psi(\operatorname{CF}(\{ y_k \}, \delta))(w^{\prime}) < \psi(\operatorname{CF}(\{ x \}, \varepsilon + \delta))(w^{\prime}) \le -\varepsilon$, then $\psi(\operatorname{CF}(\{ x \}, \varepsilon + \delta))(w^{\prime}) = -\varepsilon$.
Since $\delta$ is small and rational functions on tropical curves are continuous, there exists $w^{\prime} \in \Gamma_2$ such that $\psi(\operatorname{CF}(\{ y_1 \}, \delta))(w^{\prime}) = \psi(\operatorname{CF}(\{ y_2 \}, \delta))(w^{\prime}) = - \delta$ and $-\varepsilon - \frac{\delta}{n} < \psi(\operatorname{CF}(\{ x \}, \varepsilon + \delta))(w^{\prime}) < -\varepsilon$.
However, by the above equality, such $w^{\prime}$ cannot exist.
Therefore we have the conclusion.
\end{proof}

\begin{cor}
	\label{cor5}
Let $\psi$ and $\varphi$ be as in Theorem \ref{thm1}.
Then 
\begin{align*}
\psi (f) = f \circ \varphi
\end{align*}
holds for any $f \in \operatorname{Rat}(\Gamma_1)$.
\end{cor}

In the proof of Corollary \ref{cor5}, we will use the proof of the main theorem of \cite{JuAe3} ``for a tropical curve $\Gamma$, $\operatorname{Rat}(\Gamma)$ is finitely generated as a semifield over $\boldsymbol{T}$";
in that proof, the author gave a finite generating set of $\operatorname{Rat}(\Gamma)$ as follows.
Let $\Gamma^{\prime}$ be a metric graph that is obtained from $\Gamma$ by contracting all edges of length $\infty$.
For the canonical model $(G^{\prime}_{\circ}, l^{\prime}_{\circ})$ for $\Gamma^{\prime}$, fix a direction on edges of $G^{\prime}_{\circ}$.
We identify each edge $e^{\prime} \in E(G^{\prime}_{\circ})$ with the interval $[0, l^{\prime}_{\circ}(e)]$ with this direction.
For each edge $e^{\prime} \in E(G^{\prime}_{\circ})$, let $x_{e^{\prime}} = \frac{l^{\prime}_{\circ}(e^{\prime})}{4}$, $y_{e^{\prime}} = \frac{l^{\prime}_{\circ}(e^{\prime})}{2}$, and $z_{e^{\prime}} = \frac{3l^{\prime}_{\circ}(e^{\prime})}{4}$.
Let 
\begin{align*}
f^{\prime}_e := \operatorname{CF}\left( \{ y_{e^{\prime}} \}, \frac{l^{\prime}_{\circ}(e^{\prime})}{2} \right), g^{\prime}_e := \operatorname{CF}\left( \{ x_{e^{\prime}} \}, \frac{l^{\prime}_{\circ}(e^{\prime})}{4} \right), h^{\prime}_e := \operatorname{CF}\left( \{ z_{e^{\prime}} \}, \frac{l^{\prime}_{\circ}(e^{\prime})}{4} \right).
\end{align*}
The natural inclusion $\iota : \Gamma^{\prime} \hookrightarrow \Gamma$ induces the natural inclusion $\kappa : \operatorname{Rat}(\Gamma^{\prime}) \hookrightarrow \operatorname{Rat}(\Gamma)$ such that for any $f^{\prime} \in \operatorname{Rat}(\Gamma^{\prime})$ and $x^{\prime} \in \Gamma^{\prime}$, $\kappa(f^{\prime})(\iota(x^{\prime})) = f^{\prime}(x^{\prime})$ and $\kappa(f^{\prime})$ is extended to be constant on each connected component of $\Gamma \setminus \iota(\Gamma^{\prime})$.
Let $L_1, \ldots, L_m$ be all the connected components of $\Gamma \setminus \iota(\Gamma^{\prime})$.
Then $\{ \kappa(f_{e^{\prime}}), \kappa(g_{e^{\prime}}), \kappa(h_{e^{\prime}}), \kappa(\operatorname{CF}(\{ v^{\prime} \}, \infty)), \operatorname{CF}(\Gamma \setminus L_1, \infty), \ldots, \operatorname{CF}(\Gamma \setminus L_m, \infty) \,|\, e^{\prime} \in E(G^{\prime}_{\circ}), v^{\prime} \in V(G^{\prime}_{\circ}) \}$ is the desired generating set (see \cite[Section 1 and the proof of Lemma 1.4]{JuAe3}).
Note that we chose the canonical model for $\Gamma^{\prime}$ so that the generating set is as small as possible, but the same proof holds for any model for $\Gamma^{\prime}$.

\begin{proof}[Proof of Corollary \ref{cor5}]
Let $x \in \Gamma_1 \setminus \Gamma_{1\infty}$ and $\varepsilon > 0$ be such that $\psi(\operatorname{CF}(\{ x \}, \varepsilon))$ satisfies all of the conditions in Claim \ref{cl3}.
Let $y \in \Gamma_{1\infty}$ and $z$ be a finite point on the unique edge incident to $y$ such that $\psi(\operatorname{CF}(\Gamma_1 \setminus (z, y], \infty))^{\odot(-1)}$ satisfies all of the conditions in Claim \ref{cl8}.
Let $f := \operatorname{CF}(\{ x \}, \varepsilon)$ and $g := \operatorname{CF}(\Gamma_1 \setminus (z, y], \infty)^{\odot(-1)}$.
By Lemma \ref{lem7} and Claims \ref{cl6}, \ref{cl12}, \ref{cl13},  and the definition of $\varphi$, we have
\begin{align*}
\psi(f) = f \circ \varphi \quad \text{and} \quad \psi(g) = g \circ \varphi.
\end{align*}
Since all of such $f$ and $g$ generate $\operatorname{Rat}(\Gamma_1)$ as a tropical semifield over $\boldsymbol{T}$, for any $h \in \operatorname{Rat}(\Gamma_1)$, we have $\psi(h) = h \circ \varphi$.
In fact, if $\psi(\widetilde{f}) = \widetilde{f} \circ \varphi$ and $\psi(\widetilde{g}) = \widetilde{g} \circ \varphi$, then we have
\begin{align*}
\psi(\widetilde{f} \oplus \widetilde{g}) =&~ \psi(\widetilde{f}) \oplus \psi(\widetilde{g}) \\
=&~ \widetilde{f} \circ \varphi \oplus \widetilde{g} \circ \varphi\\
=&~ (\widetilde{f} \oplus \widetilde{g}) \circ \varphi
\end{align*}
and
\begin{align*}
\psi(\widetilde{f} \odot \widetilde{g}) =&~ \psi(\widetilde{f}) \odot \psi(\widetilde{g}) \\
=&~ \widetilde{f} \circ \varphi \odot \widetilde{g} \circ \varphi\\
=&~ (\widetilde{f} \odot \widetilde{g}) \circ \varphi.
\end{align*}
For each $y \in \Gamma_{1\infty}$, fix $z$ as above.
Let $\Gamma_{11}$ be the metric graph obtained from $\Gamma_1$ by contracting all edges of length $\infty$ to each $z$ and $\kappa : \operatorname{Rat}(\Gamma_{11}) \hookrightarrow \operatorname{Rat}(\Gamma_1)$ as above.
By lemma \ref{lemma9}, for a model $(G_1, l_1)$ for $\Gamma_{11}$ such that $V(G_1)$ contains $V \setminus \Gamma_{1\infty}$, where $V$ is as in the proof of Claim \ref{cl15}, and any edge $e \in E(G_1)$, we can choose $\kappa(f_e)$, $\kappa(g_e)$, $\kappa(h_e)$ as $f$ above.
By \cite[Remark 3.1]{JuAe2} (cf. \cite[Lemma 3.3]{JuAe2}), $\operatorname{CF}(\{ v \}, \infty)$ is contained in the semifield generated by $\{ f_e, g_e, h_e, \operatorname{CF}(\{ w \}, \varepsilon_w) \,|\, e \in E(G_1), w \in V(G_1) \}$ over $\boldsymbol{T}$ on $\Gamma_{11}$, where $\varepsilon_w$ is a positive number such that $\psi(\operatorname{CF}(\{ w \}, \varepsilon_w))$ satisfies all of the conditions in Claim \ref{cl5}.
Hence, we have the conclusion.
\end{proof}

Now we can prove Corollary \ref{cor1}:

\begin{proof}[Proof of Corollary \ref{cor1}]
Clearly, both $\mathscr{C}, \mathscr{D}$ are categories.

Let 
\begin{align*}
F: \mathscr{C} \to \mathscr{D}
\end{align*}
be
\begin{align*}
\operatorname{Ob}(\mathscr{C}) \to \operatorname{Ob}(\mathscr{D}); \qquad \Gamma \mapsto \Gamma
\end{align*}
and for $\Gamma_1, \Gamma_2 \in \operatorname{Ob}(\mathscr{C})$, 
\begin{align*}
\operatorname{Hom}_{\mathscr{C}}(\Gamma_1, \Gamma_2) \to \operatorname{Hom}_{\mathscr{D}}(\Gamma_2, \Gamma_1); \qquad \psi \mapsto \varphi,
\end{align*}
where $\varphi$ is the surjective morphism $\Gamma_2 \twoheadrightarrow \Gamma_1$ defined as in Theorem \ref{thm1}.
Let 
\begin{align*}
G: \mathscr{D} \to \mathscr{C}
\end{align*}
be
\begin{align*}
\operatorname{Ob}(\mathscr{D}) \to \operatorname{Ob}(\mathscr{C}); \qquad \Gamma \to \Gamma
\end{align*}
and for $\Gamma_1, \Gamma_2 \in \operatorname{Ob}(\mathscr{D})$, 
\begin{align*}
\operatorname{Hom}_{\mathscr{D}}(\Gamma_1, \Gamma_2) \to \operatorname{Hom}_{\mathscr{C}}(\Gamma_2, \Gamma_1); \qquad \varphi \mapsto \varphi^{\ast}.
\end{align*}
By Theorem \ref{thm1} and Proposition \ref{prop}, $F$ and $G$ are defined.
Then $F$ and $G$ are (contravariant) functors.
By Corollary \ref{cor5}, we have $G \circ F = \operatorname{id}_{\mathscr{C}}$.
Clearly, $F \circ G = \operatorname{id}_{\mathscr{D}}$ holds.
\end{proof}

\end{document}